\documentclass[12pt, reqno]{amsart}
\usepackage{hyperref}

\usepackage{cases}
\usepackage[square, comma, sort&compress, numbers]{natbib}
\usepackage{float}
\usepackage{pifont}
\usepackage{bbding}
\usepackage{mathrsfs}
\usepackage{subfigure}
\usepackage{caption}
\usepackage{geometry}
\usepackage{amsmath, amsthm, amscd, amsfonts, amssymb, graphicx, color}
\usepackage{lineno}

\newtheorem{thm}{Theorem}
\newtheorem{conj}{Conjecture}
\newtheorem{theorem}{Theorem}[section]
\newtheorem{lemma}[theorem]{Lemma}
\newtheorem{proposition}[theorem]{Proposition}
\newtheorem{corollary}[theorem]{Corollary}
\theoremstyle{definition}
\newtheorem{definition}{Definition}
\newtheorem{example}{Example}

\newtheorem{problem}{Problem}

\theoremstyle{remark}
\numberwithin{equation}{section}
\DeclareMathOperator{\RE}{Re}
\DeclareMathOperator{\IM}{Im}

\newcommand{\ID}{\mathbb{D}}
\newcommand{\IC}{{\mathbb C}}

\newcommand{\dist}{{\operatorname{dist}}}

\newcounter{minutes}
\divide\time by 60
\newcounter{hours}
\multiply\time by 60 \addtocounter{minutes}{-\time}

\begin{document}
\thispagestyle{empty} \setcounter{page}{1}

\title[Some properties of univalent log-harmonic mappings]
{Some properties of univalent log-harmonic mappings}

\thanks{
File:~\jobname .tex,
          printed: \number\day-\number\month-\number\year,
          \thehours.\ifnum\theminutes<10{0}\fi\theminutes}


\author[Z. Liu]{ZhiHong Liu }
\address{Z. Liu\vskip.03in College of Science, Guilin University of Technology,
Guilin 541004, Guangxi, People's Republic of China.
}
\email{\textcolor[rgb]{0.00,0.00,0.84}{liuzhihongmath@163.com}}

\author[S. Ponnusamy]{Saminathan Ponnusamy}
\address{S. Ponnusamy,  Department of Mathematics, Indian Institute of Technology Madras,
Chennai-600 036, India.
}
\email{\textcolor[rgb]{0.00,0.00,0.84}{samy@iitm.ac.in}}


\date{\today}



\begin{abstract}
We determine the representation theorem, distortion theorem, coefficients estimate and Bohr's radius for log-harmonic starlike mappings of order $\alpha$,
which are generalization of some earlier results. In addition, the inner mapping radius of log-harmonic mappings is also established by constructing a family of
$1$-slit log-harmonic mappings. Finally, we introduce pre-Schwarzian, Schwarzian derivatives and Bloch's norm for non-vanishing log-harmonic mappings,
several properties related to these  are also obtained.
\end{abstract}
\maketitle

\section{Introduction}

Let $\mathcal{B}$ denote the set of all bounded analytic functions defined on the unit disk $\ID=\{z: |z|<1\}$
satisfying $|\omega (z)|<1$ for all $z\in \mathbb{D}$. Then the differential operators
\begin{equation*}
\frac{\partial}{\partial z}=\frac{1}{2}\left(\frac{\partial}{\partial x}-i \frac{\partial}{\partial y}\right)\quad{\rm and}\quad
\frac{\partial}{\partial \overline{z}}=\frac{1}{2}\left(\frac{\partial}{\partial x}+i \frac{\partial}{\partial y}\right)
\end{equation*}
show that the Laplacian is given by
$$\Delta=4\frac{\partial^{2}}{\partial
z\partial\overline{z}}=\frac{\partial ^2}{\partial x^2} +
\frac{\partial ^2}{\partial y^2}.
$$
Thus a $C^2$-function $f$ defined on the unit disk $\ID$ is said to be harmonic in $\ID$ if $\Delta f=0$ therein.
Analogously, a log-harmonic mapping defined on $\ID$ is a solution of the nonlinear elliptic partial differential equation
\begin{equation*}
\frac{\overline{f_{\overline{z}}}}{\overline{f}}=\mu\frac{f_{z}}{f},
\end{equation*}
for some $\mu \in \mathcal B$, where $\mu$ is called the \textit{second complex-dilatation} of $f$. It follows that the Jacobian
$$J_{f}=|f_z|^2-|f_{\overline{z}}|^2=|f_z|^2(1-|\mu|^2)
$$
is positive and all non-constant log-harmonic mappings are therefore sense-preserving and open in $\ID$. If $f$ does not vanish in $\ID$, then $f$ can be expressed as
\begin{equation*}
f(z)=h(z)\overline{g(z)},
\end{equation*}
where $h$ and $g$ are analytic in $\ID$.
On the other hand, if $f$ is a non-constant log-harmonic mapping that vanishes only at $z=0$, then $f$ admits the representation
\begin{equation}\label{SLH1}
f(z)=z^m|z|^{2\beta m}h(z)\overline{g(z)},
\end{equation}
where $m$ is a non-negative integer, $\RE \beta>-1/2$, $h$ and $g$ are analytic in $\ID$ satisfying $h(0)\neq 0$ and $g(0)=1$
(see \cite{Abdulhadi1988}). We see that $\beta$ in \eqref{SLH1} depends only on $\mu(0)$ and can be expressed as
\begin{equation*}
  \beta=\overline{\mu(0)}\frac{1+\mu(0)}{1-|\mu(0)|^2}.
\end{equation*}

Note that $f(0)\neq 0$ if and only if $m=0$, and that a univalent log-harmonic mapping in $\ID$ vanishes at the origin if and only if $m=1$.
In other words, every univalent log-harmonic mapping in $\ID$ which vanishes at the origin has the form
\begin{equation*}
f(z)=z|z|^{2\beta}h(z)\overline{g(z)},
\end{equation*}
where $\RE \beta>-1/2$ and $0\not\in (hg)(\ID)$. The class of such functions has been widely studied.
See for example \cite{Abdulhadi1987,Abdulhadi1988,Abdulhadi1993,Abdulhadi1996}.

In this paper, our emphasis is primarily on sense-preserving univalent log-harmonic mappings in $\ID$ with $\mu(0)=0$. These mappings have the form
\begin{equation}\label{SLH2}
f(z)=zh(z)\overline{g(z)},
\end{equation}
where $h$ and $g$ are analytic in $\ID$ such that
\begin{equation}\label{LHBJC}
h(z)=\exp\left(\sum_{n=1}^{\infty}a_{n}z^n\right)~\mbox{ and }~g(z)=\exp\left(\sum_{n=1}^{\infty}b_{n}z^n\right).
\end{equation}
Here $h(z)$ and $g(z)$ may be called as analytic and co-analytic factors of $f(z)$.
Denote by $\mathcal{S}_{Lh}$ the class which consists of all such mappings.

It follows from ~\eqref{SLH2} that the functions $h,g$ and the dilatation $\mu$ satisfy
\begin{equation}\label{LHD}
\mu(z)=\frac{zg'(z)/g(z)}{1+zh'(z)/h(z)}=\frac{z\left(\log g \right)'(z)}{1+z\left(\log h \right)'(z)}.
\end{equation}

We say that a univalent log-harmonic mapping $f$ of the form \eqref{SLH2} is log-harmonic starlike mapping of order $\alpha$, denoted by
$f\in\mathcal{S}^{*}_{Lh}(\alpha)$, if
\begin{equation*}
\frac{\partial}{\partial\theta} \left(\arg f(re^{i\theta})\right)
=\RE\left(\frac{Df(z)}{f(z)}\right)
=\RE\left(\frac{zf_z(z)-\overline{z}f_{\overline{z}}(z)}{f(z)}\right)>\alpha,
\end{equation*}
for all $z=re^{i\theta}\in\ID\backslash\{0\}$ and for some $0\leq\alpha< 1$. If $\alpha=0$,
then we get the class of log-harmonic starlike mappings, $\mathcal{S}^{*}_{Lh}(0)=:\mathcal{S}^{*}_{Lh}$.
If $f$ is analytic in $\ID$, then denote by $\mathcal{S}^{*}(\alpha)$ the class of analytic starlike function
of order $\alpha$, and $\mathcal{S}^{*}(0)=:\mathcal{S}^{*}$.

The following theorem establishes a link between the classes $\mathcal{S}^{*}_{Lh}(\alpha)$ and $\mathcal{S}^{*}(\alpha)$.

\begin{thm}\label{thmLSS}
{\rm(\cite[Lemma 2.4]{Abdulhadi1987} and~\cite[Theorem 2.1]{Abdulhadi2006})}
Let $f(z)=zh(z)\overline{g(z)}$ be a log-harmonic mapping on $\ID$, $0\not\in (hg)(\ID)$. Then $f\in\mathcal{S}^{*}_{Lh}(\alpha)$ if and only if
$\varphi\in\mathcal{S}^{*}(\alpha)$, where $\varphi(z)=zh(z)/g(z)$.
\end{thm}

In \cite{Abdulhadi2006},   Abdulhadi and  Abumuhanna obtained the following representation theorem and distortion theorem for functions in $\mathcal{S}^{*}_{Lh}(\alpha)$.
\begin{thm}\label{thmfexpr}
{\rm(\cite[Theorem 2.2]{Abdulhadi2006})}
$f(z)=zh(z)\overline{g(z)}\in\mathcal{S}^{*}_{Lh}(\alpha)$ with $\mu(0)=0$ if and only if there are two probability measures $\delta$ and $\kappa$ such that
\begin{equation*}
f(z)=z\exp\left(\int_{\partial\ID}\int_{\partial\ID}K(z,\eta,\xi)d\delta(\eta)d\kappa(\xi)\right),
\end{equation*}
where
\begin{equation*}
K(z,\eta,\xi)=(1-\alpha)\log\left(\frac{1+\overline{\xi z}}{1-\eta z}\right)+T(z,\eta,\xi).
\end{equation*}
Here
\begin{small}
\begin{equation*}
T(z,\eta,\xi)=\left\{
\begin{aligned}
         &-2(1-\alpha)\IM\left(\frac{\eta+\xi}{\eta-\xi}\right)
         \arg\left(\frac{1-\xi z}{1-\eta z}\right)-2\alpha\log|1-\xi z|&{\rm if}\, |\eta|=|\xi|=1, \eta\neq \xi, \\
         &(1-\alpha)\RE\left(\frac{4\eta z}{1-\eta z}\right)-2\alpha\log|1-\eta z| &{\rm if}\, |\eta|=|\xi|=1, \eta= \xi.
\end{aligned}
\right.
\end{equation*}
\end{small}
\end{thm}

\begin{thm}\label{fDthm}
{\rm(\cite[Theorem 3.1]{Abdulhadi2006})}
Let $f(z)=zh(z)\overline{g(z)}\in\mathcal{S}^{*}_{Lh}(\alpha)$ with $\mu(0)=0$. Then for $z\in \ID$ we have
\begin{equation*}
\frac{|z|}{(1+|z|)^{2\alpha}}\exp\left((1-\alpha)\frac{-4|z|}{1+|z|}\right)\leq|f(z)|\leq \frac{|z|}{(1-|z|)^{2\alpha}}\exp\left((1-\alpha)\frac{4|z|}{1-|z|}\right).
\end{equation*}
The equalities occur if and only if $f(z)$ is one of the functions of the form $\overline{\eta} f_{\alpha}(\eta z), |\eta|=1$, where $f_{\alpha}(z)$ is given by
\begin{equation}\label{Dfalpha}
f_{\alpha}(z)=\frac{z}{1-z}\frac{1}{(1-\overline{z})^{2\alpha-1}}
\exp\left((1-\alpha)\RE\left(\frac{4z}{1-z}\right)\right).
\end{equation}
\end{thm}

The  paper is organized as follows. In Section \ref{sect2}, we present the representation theorem for the analytic and the co-analytic
products of $\mathcal{S}^{*}_{Lh}(\alpha)$ and use them to  derive distortion theorems, coefficients estimates and Bohr's radius.
   In Section \ref{sect4},
the inner mapping radius of log-harmonic mappings is established by constructing a family of 1-slit log-harmonic mappings
and propose a problem of the inner mapping radius for log-harmonic mappings. In Section \ref{sect3}, we introduce
pre-Schwarzian, Schwarzian derivatives and log-harmonic Bloch mappings and 
in Section \ref{sect6}, we continue to discuss the log-harmonic Bloch space $\mathcal{B}_{Lh}$ of non-vanishing log-harmonic mappings.

\section{Coefficients estimate}\label{sect2}



In order to prove our main results, we shall need the following lemmas.

\begin{lemma}{\rm(\cite[Corollary 3.6]{Hallenbeck1984})}\label{lempm}
Let $p(z)$ be analytic in $\mathbb{D}$ with $p(0)=1$. Then $\RE p(z)>0$ in $\mathbb{D}$ if and only if
there is a probability measure $\delta$ on $\partial\mathbb{D}$ such that
\begin{equation*}
p(z)=\int_{\partial\ID}\frac{1+\eta z}{1-\eta z}\,d\delta(\eta),\quad z\in \ID.
\end{equation*}
\end{lemma}

Since each $p$ has the form
$$ p(z)=\frac{1+\mu(z)}{1-\mu(z)}=1+\frac{2\mu(z)}{1-\mu(z)}
$$
for some $\mu\in \mathcal{B}$,  we have the following equivalent version of Lemma \ref{lempm}.

\begin{lemma}
\label{lemmu}
If $\mu\in\mathcal{B}$ with $\mu(0)=0$, then
\begin{equation*}
\frac{\mu(z)}{1-\mu(z)}=\int_{\partial\ID}\frac{\xi z}{1-\xi z}\,d\kappa(\xi),\quad z\in\ID,
\end{equation*}
for some probability measure $\kappa$ on $\partial\ID$.
\end{lemma}

\begin{theorem}
A log-harmonic mapping $f(z)=zh(z)\overline{g(z)}\in \mathcal{S}^{*}_{Lh}(\alpha)$ if and only if there are two probability measures $\delta$ and
$\kappa$ on $\partial\ID$ such that
\begin{equation}\label{eqh}
h(z)=\exp\left(\int_{\partial\ID}\int_{\partial\ID}K_1(z,\eta,\xi)\,d\delta(\eta)\,d\kappa(\xi)\right),
\end{equation}
where
\begin{equation}\label{eqh1}
K_1(z,\eta,\xi)=\left\{
\begin{aligned}
         & \left(\frac{(1-2\alpha)\eta+\xi}{\eta-\xi}
         \log\left(\frac{1-\xi z}{1-\eta z}\right)-\log(1-\eta z)\right) &{\rm if}\, |\eta|=|\xi|=1, \eta\neq \xi, \\
         &\frac{2(1-\alpha)\eta z}{1-\eta z}-\log(1-\eta z) &{\rm if}\, |\eta|=|\xi|=1, \eta= \xi,
\end{aligned}
\right.
\end{equation}
and
\begin{equation}\label{eqg}
g(z)=\exp\left(\int_{\partial\ID}\int_{\partial\ID}K_2(z,\eta,\xi)\,d\delta(\eta)\,d\kappa(\xi)\right),
\end{equation}
where
\begin{equation*}
K_2(z,\eta,\xi)= \left\{
\begin{aligned}
         & \frac{(1-2\alpha)\eta+\xi}{\eta-\xi}
         \log\left(\frac{1-\xi z}{1-\eta z}\right)+(1-2\alpha)\log(1-\eta z)  &{\rm if}\,\eta\neq\xi, \\
         & \frac{2(1-\alpha)\eta z}{1-\eta z}+(1-2\alpha)\log(1-\eta z)  &{\rm if}\, \eta=\xi.
\end{aligned}
\right.
\end{equation*}
\end{theorem}

\begin{proof}
The proof could be extracted from Theorem \ref{thmfexpr} after some computation. Because of it is independent interest and use in our investigation,
we need explicit representation for  the analytic and co-analytic factors $h(z)$ and $g(z)$ of $f(z)$, and thus we include the proof.

According to Theorem \ref{thmLSS}, we see that $f(z)=zh(z)\overline{g(z)}\in \mathcal{S}^{*}_{Lh}(\alpha)$ if and only if
$\varphi(z)=zh(z)/g(z)\in\mathcal{S}^{*}(\alpha)$, i.e.,
\begin{equation*}
\frac{z\varphi'(z)}{\varphi(z)}=(1-\alpha)p(z)+\alpha,
\end{equation*}
where $p$ is analytic in $\ID$ such that $p(0)=1$ and $\RE p(z)>0$ in $\ID$. Thus, by Lemma \ref{lempm}, it follows that
\begin{equation}\label{vrpf}
\frac{z\varphi'(z)}{\varphi(z)}=(1-\alpha)\int_{\partial\ID}\frac{1+\eta z}{1-\eta z}\,d\delta(\eta)+\alpha,
\end{equation}
and therefore, we have the following well-known representation for $\varphi(z)\in\mathcal{S}^{*}(\alpha)$:
\begin{equation}\label{vrpE}
\varphi(z)=z \exp\left(-2(1-\alpha)\int_{\partial\ID}\log(1-\eta z)\,d\delta(\eta)\right).
\end{equation}
From \eqref{LHD}, \eqref{vrpf} and Lemma \ref{lemmu}, it follows that
\begin{equation}\label{Exprg}
\begin{split}
g(z)&=\exp\left(\int_{0}^{z}\left(\frac{\mu(s)}{1-\mu(s)}
\cdot\frac{\varphi'(s)}{\varphi(s)}\right)\,ds\right)
\end{split}
\end{equation}
so that
\begin{equation}\label{Exprg1}
\begin{split}
g(z)=\exp\left(\int_{0}^{z}\int_{\partial\ID}\int_{\partial\ID}\frac{\xi}{1-\xi s}\left((1-\alpha)\frac{1+\eta s}{1-\eta s}+\alpha\right)\,d\delta(\eta)\,d\kappa(\xi)\,ds\right)
\end{split}
\end{equation}
for some probability measures $\delta$ and $\kappa$ on $\partial\ID$.

Moreover, if $\eta\neq\xi$, we may reduce $g$ in the form
\begin{small}
\begin{equation*}
\begin{split}
g(z)&=\exp\left(\int_{\partial\ID}\int_{\partial\ID}\int_{0}^{z}\frac{\xi}{1-\xi s}\left((1-\alpha)\frac{1+\eta s}{1-\eta s}+\alpha\right)\,ds\,d\delta(\eta)\,d\kappa(\xi)\right)\\
&=\exp\left(\int_{\partial\ID}\int_{\partial\ID}\left(\frac{(1-2\alpha)\eta+\xi}{\eta-\xi}\log(1-\xi z)-\frac{2(1-\alpha)\xi}{\eta-\xi}\log(1-\eta z)\right)\,d\delta(\eta)\,d\kappa(\xi)\right)\\
&=\exp\left(\int_{\partial\ID}\int_{\partial\ID}\left(\frac{(1-2\alpha)\eta+\xi}{\eta-\xi}\log\left(\frac{1-\xi z}{1-\eta z}\right)+(1-2\alpha)\log(1-\eta z)\right)\,d\delta(\eta)\,d\kappa(\xi)\right).
\end{split}
\end{equation*}
\end{small}
On the other hand, if $\eta=\xi$, then from \eqref{Exprg1} we see that
\begin{equation*}
\begin{split}
g(z)&=\exp\left(\int_{\partial\ID}\int_{\partial\ID}\int_{0}^{z}\frac{\eta}{1-\eta s}\left((1-\alpha)\frac{1+\eta s}{1-\eta s}+\alpha\right)\,ds\,d\delta(\eta)\,d\kappa(\eta)\right)\\
&=\exp\left(\int_{\partial\ID}\int_{\partial\ID}\left(\frac{2(1-\alpha)\eta z}{1-\eta z}+(1-2\alpha)\log(1-\eta z)\right)\,d\delta(\eta)\,d\kappa(\eta)\right).
\end{split}
\end{equation*}

Finally, by writing
\begin{equation*}
\begin{split}
h(z)=\frac{\varphi(z)}{z}g(z)
\end{split}
\end{equation*}
the representation for $h$ can easily be obtained from the last expression for $g$ and \eqref{vrpE}.
\end{proof}

\begin{theorem}\label{thmSLHA}
Let $f(z)=zh(z)\overline{g(z)}\in\mathcal{S}^{*}_{Lh}(\alpha)$ with $\mu(0)=0$. Then for $z\in\ID$,
\begin{enumerate}
\item $\frac{1}{1+|z|}\exp\left((1-\alpha)\frac{-2|z|}{1+|z|}\right)\leq|h(z)|\leq
    \frac{1}{1-|z|}\exp\left((1-\alpha)\frac{2|z|}{1-|z|}\right)$;\label{Bh}
\item $\frac{1}{(1+|z|)^{2\alpha-1}}\exp\left((1-\alpha)\frac{-2|z|}{1+|z|}\right)\leq|g(z)|\leq \frac{1}{(1-|z|)^{2\alpha-1}}\exp\left((1-\alpha)\frac{2|z|}{1-|z|}\right)$.\label{Bg}
\end{enumerate}
The equalities occur if and only if $f(z)$ is one of the functions of the form
$\overline{\eta} f_{\alpha}(\eta z), |\eta|=1$, where $f_{\alpha}(z)$ is given by \eqref{Dfalpha}
\end{theorem}
\begin{proof}
Let $\varphi(z)=z h(z)/g(z)\in\mathcal{S}^{*}(\alpha)$ so that
\begin{equation*}
h(z)=\frac{\varphi(z)}{z}g(z)\quad{\rm and}\quad f(z)=\varphi(z)|g(z)|^2.
\end{equation*}
For $|z|=r<1$, by Theorem \ref{thmLSS}, we know that
\begin{equation*}
\left|\frac{z\varphi'(z)}{\varphi(z)}\right|\leq (1-\alpha)\frac{1+r}{1-r}+\alpha.
\end{equation*}
Because $\mu\in \mathcal{B}$ with $\mu(0)=0$, we have
\begin{equation*}
\left|\frac{\mu(z)}{z(1-\mu(z))}\right|\leq\frac{1}{1-r} \quad{\rm and}\quad |\varphi(z)|\leq \frac{r}{(1-r)^{2(1-\alpha)}},
\end{equation*}
which by \eqref{vrpf} and \eqref{Exprg} imply that
\begin{equation*}
\begin{split}
|g(z)|&\leq\exp\left(\int_{0}^{r}\frac{1}{1-s}\left[(1-\alpha)\frac{1+s}{1-s}+\alpha\right]ds\right)\\
&=\exp\left((1-\alpha)\frac{2r}{1-r}-(2\alpha-1)\log(1-r)\right)\\
&=\frac{1}{(1-r)^{2\alpha-1}}\exp\left((1-\alpha)\frac{2r}{1-r}\right).
\end{split}
\end{equation*}
The known estimate for $\varphi\in S^{*}(\alpha)$ and the last inequality give
\begin{equation*}
\begin{split}
|h(z)|&=\left|\frac{\varphi(z)}{z}\right||g(z)|\\
&\leq \frac{1}{(1-r)^{2(1-\alpha)}}
\cdot\frac{1}{(1-r)^{2\alpha-1}}\exp\left((1-\alpha)\frac{2r}{1-r}\right)\\
&=\frac{1}{1-r}\exp\left((1-\alpha)\frac{2r}{1-r}\right).
\end{split}
\end{equation*}
Equality occurs if and only if $\mu(z)=\eta z$ and $\varphi(z)=\frac{z}{(1-z)^{2(1-\alpha)}},\,|\eta|=1$, which leads to
$f(z)=\overline{\eta} f_{\alpha}(\eta z)$, where $f_{\alpha}(z)$ is given by \eqref{Dfalpha}.

For the left side estimates of \eqref{Bg}, by \eqref{eqh}, we obtain that
\begin{equation*}
\log|h(z)|=\RE\left(\int_{\partial\ID}\int_{\partial\ID}K_1(z,\xi,\eta)\,d\delta(\eta)\,
d\kappa(\xi)\right),
\end{equation*}
where $K_1(z,\xi,\eta)\,d\delta(\eta)$ is defined by \eqref{eqh1} and may be rewritten as
\begin{small}
\begin{equation*}
K_1(z,\xi,\eta)=\left\{
\begin{aligned}
         &(1-\alpha)\frac{\eta+\xi}{\eta-\xi}
         \log\left(\frac{1-\xi z}{1-\eta z}\right)-\alpha\log(1-\xi z)-(1-\alpha)\log(1-\eta z) &{\rm if}\, \eta\neq\xi, \\
         &\frac{2(1-\alpha)\eta z}{1-\eta z}-\log(1-\eta z) &{\rm if}\, \eta=\xi,
\end{aligned}
\right.
\end{equation*}
\end{small}
and $|\eta|=|\xi|=1$,
Then for $|z|=r$, we have
\begin{small}
\begin{equation*}
\begin{split}
\log|h(z)|&=\RE\left(\int_{\partial\ID}\int_{\partial\ID}K_1(z,\xi,\eta)\,d\delta(\eta)\,d\kappa(\xi)\right)\\
&\geq \min_{\delta,\,\kappa}\left\{\min_{|z|=r}\RE\left(\int_{\partial\ID}\int_{\partial\ID}K_1(z,\xi,\eta)
\,d\delta(\eta)\,d\kappa(\xi)\right)\right\}\\
&=\min\bigg\{\min_{|z|=r}\,\inf_{0<|l|\leq \pi/2}\left[-(1-\alpha)\IM\left(\frac{1+e^{2il}}{1-e^{2il}}\right)\arg\left(\frac{1-e^{2il}(\eta z)}{1-\eta z}\right)\right]-\log(1+r),\\
&\qquad\qquad(1-\alpha)\frac{-2r}{1+r}-\log(1+r)\bigg\},
\end{split}
\end{equation*}
\end{small}
where $e^{2il}=\overline{\eta}\xi$.  Now, we let
\begin{small}
\begin{equation*}
\Phi_r(l)=\left\{
\begin{aligned}
&\min_{|z|=r}\left[-(1-\alpha)\IM\left(\frac{1+e^{2il}}{1-e^{2il}}\right)
\arg\left(\frac{1-e^{2il}(\eta z)}{1-\eta z}\right)\right]-\log(1+r) &{\rm if}\, 0<|l|<\pi/2, \\
&(1-\alpha)\frac{-2r}{1+r}-\log(1+r) &{\rm if}\, l=0.
\end{aligned}
\right.
\end{equation*}
\end{small}

In a manner similar in the proof of \cite[Theorem 2]{Ali2016}, we see that the function $\Phi_r(l)$ is
continuous and is even in the interval $|l|\leq \pi/2$. Hence
\begin{equation*}
\begin{split}
\log|h(z)|&\geq \inf_{0\leq|l|\leq \pi/2}\Phi_r(l)=(1-\alpha)\frac{-2r}{1+r}-\log(1+r).
\end{split}
\end{equation*}

For the lower bound of $|g(z)|$ in Theorem \ref{thmSLHA}\eqref{Bg}, a similar discussion applied to \eqref{eqg} yields
\begin{equation*}
\begin{split}
\log|g(z)|&\geq \inf_{0\leq|l|\leq \pi/2}\left(\Phi_r(l)+2(1-\alpha)\log(1+r)\right)
=(1-\alpha)\frac{-2r}{1+r}-(2\alpha-1)\log(1+r).
\end{split}
\end{equation*}
The proof is complete.
\end{proof}


\begin{corollary}\label{thCLB}
Let $f(z)=zh(z)\overline{g(z)}\in \mathcal{S}^{*}_{Lh}(\alpha)$. Also, let $H(z)=zh(z)$ and $G(z)=zg(z)$. Then
\begin{enumerate}
\item $\frac{1}{2\,e^{1-\alpha}}\leq d(0,\partial H(\ID))\leq 1$;
\item $\frac{1}{2^{2\alpha-1}\,e^{1-\alpha}}\leq d(0,\partial G(\ID))\leq 1$;
\item $\frac{1}{2^{2\alpha}\,e^{2(1-\alpha)}}\leq d(0,\partial f(\ID))\leq 1$.
\end{enumerate}
The equalities occur if and only if $f(z)$ is one of the functions of the form $\overline{\eta} f_{\alpha}(\eta z), |\eta|=1$,
where $f_{\alpha}(z)$ is given by \eqref{Dfalpha}.
\end{corollary}

\begin{proof}
By Theorem~\ref{thmSLHA},
$$d(0,\partial H(\ID))=\liminf _{|z|\to1}|H(z)-H(0)|=\liminf _{|z|\to1}\frac{|H(z)-H(0)|}{|z|}=\liminf _{|z|\to 1}|h(z)|\geq \frac{1}{2\,e^{1-\alpha}}.$$
On the other hand, the minimum modulus principle shows that
$$d(0,\partial H(\ID))=\liminf _{|z|\to 1}|h(z)|\leq 1,
$$
since $|h(0)|=1$. The same approach may be applied to $G(z)$ and $f(z)$ to find proofs of the remaining inequalities.
\end{proof}

Now, we give a sharp upper bound for the coefficients of $h(z)$ and $g(z)$.
\begin{theorem}\label{Cfalp}
Let $f(z)=zh(z)\overline{g(z)}\in \mathcal{S}^{*}_{Lh}(\alpha)$. Then
\begin{equation}
|a_n|\leq 2(1-\alpha)+\frac{1}{n}\quad{\rm and}\quad |b_n| \leq 2(1-\alpha)+\frac{2\alpha-1}{n}
\end{equation}
for all $n\geq 1$. The equalities occur if and only if $f(z)$ is one of the functions of the form $\overline{\eta} f_{\alpha}(\eta z), |\eta|=1$,
where $f_{\alpha}(z)$ is given by \eqref{Dfalpha}.
\end{theorem}
\begin{proof}
From \eqref{eqh} and \eqref{eqg}, we get the following expressions
\begin{equation*}
\begin{split}
a_n&=\frac{1}{n}\int_{\partial \ID}\int_{\partial \ID}\left(\eta^n+\frac{(1-2\alpha)\eta+\xi}{\eta-\xi}\left(\eta^n-\xi^n\right)\right) d\delta(\eta)d\kappa(\xi)\\
&=\frac{1}{n}\int_{\partial \ID}\left(\eta^n+\int_{\partial \ID}\left(\left((1-2\alpha)\eta+\xi\right)\sum_{k=0}^{n-1}\eta^{n-k-1}\xi^{k} \right)d\kappa(\xi)\right) d\delta(\eta)
\end{split}
\end{equation*}
and
\begin{equation*}
\begin{split}
b_n&=\frac{1}{n}\int_{\partial \ID}\int_{\partial \ID}\left((2\alpha-1)\eta^n+\frac{(1-2\alpha)\eta+\xi}{\eta-\xi}\left(\eta^n-\xi^n\right)\right) d\delta(\eta)d\kappa(\xi)\\
&=\frac{1}{n}\int_{\partial \ID}\left((2\alpha-1)\eta^n+\int_{\partial \ID}\left(\left((1-2\alpha)\eta+\xi\right)\sum_{k=0}^{n-1}\eta^{n-k-1}\xi^{k} \right)d\kappa(\xi)\right) d\delta(\eta).
\end{split}
\end{equation*}
The maximum of $|a_n|$ (resp. $|b_n|$) is attained when $\delta$ and $\kappa$ are Dirac measures. Therefore, we have
\begin{equation*}
\begin{split}
|a_n|&\leq\max\left\{\frac{1}{n}\left|\eta^n+\left((1-2\alpha)\eta+\xi\right)\sum_{k=0}^{n-1}\eta^{n-k-1}\xi^{k}\right|
:\,|\eta|=|\xi|=1\right\}\\
&\leq 2(1-\alpha)+\frac{1}{n}
\end{split}
\end{equation*}
and
\begin{equation*}
\begin{split}
|b_n|&\leq\max\left\{\frac{1}{n}\left|(2\alpha-1)\eta^n+\left((1-2\alpha)\eta+\xi\right)\sum_{k=0}^{n-1}\eta^{n-k-1}\xi^{k}\right|
:\,|\eta|=|\xi|=1\right\}\\
&\leq 2(1-\alpha)+\frac{2\alpha-1}{n}.
\end{split}
\end{equation*}
The equalities occur if and only if $f(z)$ is one of the functions of the form $\overline{\eta} f_{\alpha}(\eta z), |\eta|=1$,
where $f_{\alpha}(z)$ is given by \eqref{Dfalpha},
which may be rewritten as
\begin{equation*}
\begin{split}
f_{\alpha}(z)=z\exp\left(\sum_{n=1}^{\infty}\left(2(1-\alpha)+\frac{1}{n}\right)z^n\right)
\overline{\exp\left(\sum_{n=1}^{\infty}\left(2(1-\alpha)+\frac{2\alpha-1}{n}\right)z^n\right)}.
\end{split}
\end{equation*}
This completes the proof.
\end{proof}

\section{Bohr's radius for $\mathcal{S}^{*}_{Lh}(\alpha)$ }\label{sect2a}
The classical Bohr inequality states that if $f(z)=\sum_{n=0}^{\infty}a_n z^n$ is analytic in $\ID$ and $|f(z)|\leq 1$
in $\ID$, then
\begin{equation*}
M_r(f)=\sum_{n=0}^{\infty}|a_n |r^n\leq 1
\end{equation*}
for all $|z|=r\leq 1/3$ (see Bohr \cite{Bohr1914}). Bohr actually obtained the inequality only for $|z| \leq 1/6$,
Wiener, Riesz and Schur  independently established the sharp inequality for $|z|\leq 1/3$ and showed that the bound $1/3$
was sharp. See the recent survey on this topic \cite{AAP} and the
references therein.
In recent years, the space of subordinations and the space of complex-valued
bounded harmonic mappings are considered in the study of Bohr's inequality,
for example in \cite{Muhanna2010,Muhanna2014,Kayumov2017}.

The following results concern Bohr's radius of log-harmonic starlike mappings of order $\alpha$.

\begin{theorem}\label{BohrRf}
Let $f(z)=zh(z)\overline{g(z)}\in\mathcal{S}^{*}_{Lh}(\alpha)$, $H(z)=zh(z)$ and $G(z)=zg(z)$. Then
\begin{enumerate}
\item
$\begin{aligned}[t]
     & |z|\exp\left(\sum_{n=1}^{\infty}|a_n||z|^n\right)\leq d(0,\partial H(\ID))
  \end{aligned}
$ for $|z|\leq r_H$, where $r_H$ is the unique root in $(0,1)$ of the equation
\begin{equation}\label{eq-1}
 \frac{r}{1-r}\exp\left((1-\alpha)\frac{2r}{1-r}\right)=\frac{1}{2\,e^{1-\alpha}},
\end{equation}
\item $ \begin{aligned}[t]
        & |z|\exp\left(\sum_{n=1}^{\infty}|b_n||z|^n\right)\leq d(0,\partial G(\ID))
  \end{aligned} $
for $|z|\leq r_G$, where $r_G$ is the unique root in $(0,1)$ of the equation
\begin{equation}\label{eq-2}
\frac{r}{(1-r)^{2\alpha-1}}\exp\left((1-\alpha)\frac{2r}{1-r}\right)=\frac{1}{2^{2\alpha-1}\,e^{1-\alpha}},
\end{equation}
\item $ \begin{aligned}[t]
        & |z|\exp\left(\sum_{n=1}^{\infty}\left(|a_n|+|b_n|\right)|z|^n\right)\leq d(0,\partial f(\ID))
  \end{aligned}
  $
for $|z|\leq r_f $, where $r_f$ is the unique root in $(0,1)$ of the equation
\begin{equation}\label{eq-2a}
\frac{r}{(1-r)^{2\alpha}}\exp\left((1-\alpha)\frac{4r}{1-r}\right)=\frac{1}{2^{2\alpha}\,e^{2(1-\alpha)}}.
\end{equation}
\end{enumerate}
All the radius are sharp and attained by a suitable rotation of the log-harmonic right-half plane mapping $f_{\alpha}(z)$,
where $f_{\alpha}(z)$ is given by \eqref{Dfalpha}.
\end{theorem}
\begin{proof}
By assumption
$$H(z)=z\exp\left(\sum_{n=1}^{\infty}a_n z^{n}\right)\quad{\rm and}\quad
G(z)=z\exp\left(\sum_{n=1}^{\infty}b_n z^{n}\right).$$
Firstly, we have
\begin{equation*}
\begin{split}
r\exp\left(\sum_{n=1}^{\infty}|a_n|r^n\right)
&\leq r\exp\left(\sum_{n=1}^{\infty}\left(2(1-\alpha)+\frac{1}{n}\right)r^n\right)~\mbox{ (by Theorem ~\ref{Cfalp})}\\
&=\frac{r}{1-r}\exp\left(2(1-\alpha)\frac{r}{1-r}\right)\\
&=\frac{1}{2\,e^{1-\alpha}} \leq d(0,\partial H(\ID)) ~\mbox{ (by \eqref{eq-1} and Corollary~\ref{thCLB}).}
\end{split}
\end{equation*}
Similarly, using Theorem ~\ref{Cfalp}, \eqref{eq-2} and Corollary~\ref{thCLB}, we have
\begin{equation*}
\begin{split}
r\exp\left(\sum_{n=1}^{\infty}|b_n|r^n\right)
&\leq r\exp\left(\sum_{n=1}^{\infty}\left(2(1-\alpha)+\frac{2\alpha-1}{n}\right)r^n\right)\\
&=\frac{r}{(1-r)^{2\alpha-1}}\exp\left(2(1-\alpha)\frac{r}{1-r}\right)\\
&\leq d(0,\partial G(\ID)).
\end{split}
\end{equation*}
Furthermore, using Theorem ~\ref{Cfalp}, \eqref{eq-2a} and Corollary~\ref{thCLB}, we have
\begin{equation*}
\begin{split}
r\exp\left(\sum_{n=1}^{\infty}\left(|a_n|+|b_n|\right)r^n\right)
&\leq r\exp\left(\sum_{n=1}^{\infty}\left(4(1-\alpha)+\frac{2\alpha}{n}\right)r^n\right)\\
&=\frac{r}{(1-r)^{2\alpha}}\exp\left((1-\alpha)\frac{4r}{1-r}\right)\\
&\leq d(0,\partial f(\ID)).
\end{split}
\end{equation*}

Finally, it is evident that all radius are attained by suitable rotations of the  log-harmonic right half plane mapping
$f_{\alpha}(z)$, where $f_{\alpha}(z)$ is given by \eqref{Dfalpha}.
\end{proof}

If $\alpha=0$, then Theorem \ref{BohrRf} reduces to Theorem 3 in \cite{Ali2016}. If $\alpha\to1$, then $r_H=r_G=1/3$,
and $r_f=3-2\sqrt{2}$ which is same as Bohr's radius of the subordinating family of univalent functions
(see \cite[Theorem 1]{Muhanna2010}) which we recall for a ready reference below.
\begin{thm}\label{subtheo}
Suppose that $f,g$ are analytic in $\ID$ such that $f$ is univalent in $\ID$ and $g(z)=\sum_{n=0}^{\infty} b_nz^n$
belongs to $S(f)=\{\varphi:\, \varphi \prec f\}$, where $\prec$ denotes the usual subordination. Then inequality
$$\sum_{n=1}^{\infty} |b_n|r^n \leq \dist (f(0),\partial f(\ID))
$$
holds with $r_f=3-2\sqrt{2}\approx 0.17157$. The sharpness of $r_f$ is shown by the Koebe function $f(z)=z/(1-z)^2.$
\end{thm}

\section{The inner mapping radius of log-harmonic mappings}\label{sect4}
In \cite{Liu1}, the authors have proposed the following two conjectures.

\begin{conj}\label{LHBJ}
Let $f(z)=zh(z)\overline{g(z)}\in\mathcal{S}_{Lh}$, where the representation of $h(z)$ and $g(z)$ are given by~\eqref{LHBJC}.
Then for all $n\geq 1$,
\begin{enumerate}
\item $|a_{n}|\leq 2+ \frac{1}{n}$;
\item $|b_{n}|\leq 2- \frac{1}{n}$;
\item $\left|a_{n}-b_{n}\right|\leq \frac{2}{n}$;
\end{enumerate}
\end{conj}
This conjecture has been verified for starlike log-harmonic mappings, see ~\cite[Theorem 3.3]{Abdulhadi1989} and \cite[Theorem~3.3]{Liu1}. The
log-harmonic Koebe function $f_0(z)$ given by \eqref{Dfalpha} gives the sharpness. In \cite[Theorem~3.3]{Liu1}, it was also proposed that
$\{w:\,|w|<1/e^2\}\subseteq f(\ID)$ if $f(z)=zh(z)\overline{g(z)}\in\mathcal{S}_{Lh}$.

\begin{definition}
For $f(z)=zh(z)\overline{g(z)}\in\mathcal{S}_{Lh}$ , the inner mapping radius $\rho_0(f)$ of the domain $f(\ID)$
is to be the real number $\psi'(0)$, where
$\psi$ is the analytic function that maps $\ID$ onto $f(\ID)$ such that  $\psi(0)=0$ and
$\psi'(0)>0$.
\end{definition}


Recall that the tip of the slit of the log-harmonic Koebe function $f_0(z)$ is at $-1/e^2$
while the tip of the slit for the analytic Koebe function $k(z)=\frac{z}{(1-z)^2}$ is at $-1/4$.
Obviously, the images of the unit disk under $\frac{4}{e^2} k(z)$ and under
$f_0(z)$ are the same, i.e.,
$$\frac{4}{e^2}k(\ID)=f_0(\ID).$$
This multiplier factor of $4/e^2$ is the inner mapping radius for $f_0(\ID)$. For other log-harmonic functions in $\mathcal{S}_{Lh}$, the inner mapping
radius may be different. For example, consider log-harmonic right half-plane mapping $LR(z)$ and log-harmonic two-slits mapping $LS(z)$
(see from Examples 2 and  3 in ~\cite{Liu1}) given by
\begin{equation*}
\begin{split}
LR(z)=\frac{z}{1-z}\exp\left(\RE\left(\frac{2z}{1-z}\right)\right)
\end{split}
\end{equation*}
and
\begin{equation*}
\begin{split}
LS(z)&=\frac{z}{1-z^2}|1-z^2|\exp\left(\RE\left(\frac{2z^2}{1-z^2}\right)\right),
\end{split}
\end{equation*}
respectively. The inner mapping radius for $LR(\ID)$ and $LS(\ID)$ is $1/e$, since
$$\frac{1}{e}R(\ID)=LR(\ID)\quad{\rm and}\quad\frac{1}{e}S(\ID)=LS(\ID),$$
where $R(z)=z/(1-z)$ and $S(z)=z/(1-z^2)$ denote the analytic right half-plane mapping and two-slits mapping, respectively.

In the example above $\psi(z) =\frac{4}{e^2} k(z)$, and the inner
mapping radius $\rho_0(f_0)=\frac{4}{e^2}$. In the following example, we show that $\frac{4}{e^2}\leq\rho_0(f)\leq4$ for one slit log-harmonic
mappings $f\in \mathcal{S}_{Lh}$.
\begin{example}
Consider the family of functions $F_{\lambda}(z)=f_{1}^{\lambda}(z)f_{2}^{1-\lambda}(z)~(0\leq\lambda\leq1)$, where \begin{equation*}
f_1(z)=\frac{z}{(1-z)^2}|1-z|^2\quad{\rm and}\quad f_2(z)=\frac{z}{(1-z)^2}|1-z|^2\exp\left(\RE\left(\frac{4z}{1-z}\right)\right).
\end{equation*}

Simple calculations show that $f_1$ and $f_2$ are starlike log-harmonic with dilatations $\mu_1(z)=-z$ and $\mu_2(z)=z$. Also $F_{\lambda}$ is
log-harmonic with  the dilatation
$$\mu(z)=\frac{z[(1-2\lambda)+z]}{1+(1-2\lambda)z}.
$$
It is clear that $|\mu(z)|<1$ for $0\leq\lambda\leq 1$, and therefore $F_{\lambda}$ is sense-preserving in $\ID$. Since the conditions of
Theorem 3 in~\cite{AbdulHadi2014} are satisfied (or see the details in Example 3 in~\cite{AbdulHadi2014}), we thus have that $F_{\lambda}$ is univalent and starlike in $\ID$.
Moreover,
$$F_{\lambda}(z)=f_{1}^{\lambda}(z)f_{2}^{1-\lambda}(z)
=\frac{z}{(1-z)^2}|1-z|^2\exp\left((1-\lambda)\RE\left(\frac{4z}{1-z}\right)\right).$$
Because $k(\ID)$ is $\IC$ minus the slit on the negative real axis from $-1/4$ to $\infty$, and  $F_{\lambda}(\ID)$ is $\IC$ minus
the slit on the negative real axis from $-e^{-2(1-\lambda)}$ to $\infty$, we obtain that for $0 < \lambda < 1$,
$$\frac{4}{e^2}\leq\rho_0(F_{\lambda})\leq 4.$$
\end{example}


\begin{problem}
Show that the inner mapping radius of log-harmonic mapping $f\in\mathcal{S}_{Lh}$ satisfy
$$\frac{4}{e^2}\leq\rho_0(f)\leq 4;$$
or else find a class of log-harmonic mappings $f\in\mathcal{S}_{Lh}$ such that either
$\rho_0(f)> 4$ or $\rho_0(f)<\frac{4}{e^2}.$
\end{problem}

\section{Pre-Schwarzian derivatives and log-harmonic mappings}\label{sect3}

In this section, we introduce pre-Schwarzian, Schwarzian derivatives and log-harmonic Bloch function for non-vanishing
log-harmonic mappings analogous to analytic and harmonic mappings.
%

The pre-Schwarzian and Schwarzian derivatives of a locally univalent analytic function $h$ are given (cf. \cite{Duren1983}) by
\begin{equation*}
Ph(z)=\frac{h''(z)}{h'(z)}\quad{\rm and}\quad Sh(z)=\left(\frac{h''(z)}{h'(z)}\right)'-\frac{1}{2}\left(\frac{h''(z)}{h'(z)}\right)^2,
\end{equation*}
respectively. These notions for complex valued harmonic mappings was presented by  Chuaqui et al. \cite{Chuaqui2003} and investigated by
a number of authors. See \cite{Chuaqui2007,Chuaqui2008,Hernandez2015,ChenSh2017} and the references therein.
In \cite{Mao2013},  Mao and  Ponnusamy investigated the Schwarzian derivative of log-harmonic mappings, and they obtained several necessary and sufficient
conditions for Schwarzian derivative $S_f$ to be analytic. In this paper, we modify the definitions of pre-Schwarzian $P_f$ and Schwarzian $S_f$ derivatives
for the sense-preserving univalent log-harmonic mappings and notice that the new definitions preserve the standard properties of the classical Schwarzian derivative and
they are given in the following way:
\begin{equation*}
\begin{split}
P_f(z)&=\left(\log J_f\right)_{z}=\left(\frac{h''(z)}{h'(z)}-\frac{h'(z)}{h(z)}\right)-\frac{\overline{\mu(z)}\mu'(z)}{1-|\mu(z)|^2},\\
S_f(z)&=\left(P_f(z)\right)'-\frac{1}{2}\left(P_f(z)\right)^2\\
&=\left(\frac{h''(z)}{h'(z)}-\frac{h'(z)}{h(z)}\right)'
-\frac{1}{2}\left(\frac{h''(z)}{h'(z)}-\frac{h'(z)}{h(z)}\right)^2
+\left(\frac{h''(z)}{h'(z)}-\frac{h'(z)}{h(z)}\right)\frac{\overline{\mu(z)}\mu'(z)}{1-|\mu(z)|^2}\\
&\qquad -\frac{\overline{\mu(z)}\mu''(z)}{1-|\mu(z)|^2}
-\frac{3}{2}\left(\frac{\overline{\mu(z)}\mu'(z)}{1-|\mu(z)|^2}\right)^2,
\end{split}
\end{equation*}
where
$$J_f(z)=\left|\frac{h'(z)}{h(z)}\right|^2(1-|\mu(z)|^2)~\mbox{ and }~\mu(z)=\frac{|g'(z)/g(z)|}{|h'(z)/h(z)|}
$$
are the Jacobian of log-harmonic mapping $f$ and the dilatation of $f$, respectively. The pre-Schwarzian and Schwarzian derivatives of log-harmonic mappings
have the chain rule property exactly in the same form as in the analytic case: if $f$ is a sense- preserving log-harmonic mapping and
$\varphi$ is a locally univalent analytic function for which the composition $f\circ \varphi$ is defined, then a straightforward calculation shows that
$$P_{f\circ\,\varphi}(z)=\left(P_f\circ\varphi(z)\right)\cdot \varphi'(z)+P_{\varphi}(z)
~\mbox{ and }
S_{f\circ\,\varphi}(z)=\left(S_f\circ\varphi(z)\right)\cdot (\varphi'(z))^2+S_{\varphi}(z).
$$


If we assume that the pre-Schwarzian derivative $P_f$ of a log-harmonic mapping $f=h\overline{g}$ with dilatation $\mu(z)$ is analytic,
then we get that
\begin{equation*}
\frac{\partial P_f}{\partial \overline{z}}=\frac{|\mu'(z)|^2}{(1-|\mu(z)|^2)^2}=0 \quad {(z\in\ID)},
\end{equation*}
which implies that $\mu(z)$ is constant. In other words, $P_f$ is analytic if and only if the dilatation of $f$ is constant. Actually,
we get the following more general result.

\begin{theorem}\label{PLHD}
Suppose that $f(z)=h(z)\overline{g(z)}$ is a sense-preserving log-harmonic mapping in $\ID$. Then pre-Schwarzian derivative $P_f$ of $f(z)$ is harmonic
if and only if the dilatation $\mu(z)$ of $f(z)$ is constant.
\end{theorem}
\begin{proof}
By a straightforward calculation, we obtain
\begin{equation}\label{PLf}
\frac{\partial^2 P_f}{\partial z\partial \overline{z}}=\frac{\overline{\mu'}\left(\mu''(1-|\mu|^2)+2\mu'^2\,\overline{\mu}\right)}
{(1-|\mu|^2)^3}.
\end{equation}

If $\mu$ is constant, then it is clear that $\Delta P_f\equiv0$ and so $P_f(z)$ is harmonic in $\ID$.

Now we assume that $P_f(z)$ is harmonic. By \eqref{PLf}, we get
\begin{equation*}
\overline{\mu'}\left(\mu''(1-|\mu|^2)+2\mu'^2\,\overline{\mu}\right)=0.
\end{equation*}
If $\mu$ is not constant, then the last relation reduced to
\begin{equation*}
\frac{\mu''}{\mu'^2}=-\frac{2\overline{\mu}}{1-|\mu|^2},
\end{equation*}
which is analytic in $\ID$.
Thus, we see that  $\mu$ is a constant which contradicts our assumption. The proof is complete.
\end{proof}

\section{Log-harmonic Bloch space}\label{sect6}
\begin{definition}\label{dfLHBF}
A non-vanishing log-harmonic mapping $f(z)=h(z)\overline{g(z)}$ in $\ID$ is said to be a \emph{log-harmonic Bloch function} if
$$\beta(f)=\sup_{z\in\ID}(1-|z|^2)\left(\left|\frac{h'(z)}{h(z)}\right|
+\left|\frac{g'(z)}{g(z)}\right|\right)<+\infty,$$
where $h$ and $g$ are analytic in $\ID$,
$$h(z)=\exp\left(\sum_{n=0}^{\infty}a_n z^n\right)\quad {\rm and }\quad g(z)=\exp\left(\sum_{n=1}^{\infty}b_n z^n\right).$$
The space of all log-harmonic Bloch functions is denoted by $\mathcal{B}_{Lh}$.
\end{definition}

The space $\mathcal{B}_{Lh}$ forms a complex Banach space with the norm $\|\cdot\|_{\mathcal{B}_{Lh}}$ given by~(see~\cite{Colonna1989})
\begin{equation*}
\begin{split}
\|f\|_{\mathcal{B}_{Lh}}&=|f(0)|+\sup_{z\in\ID}(1-|z|^2)\left(\left|\frac{h'(z)}{h(z)}\right|
+\left|\frac{g'(z)}{g(z)}\right|\right)\\
&=|f(0)|+\sup_{z\in\ID}(1-|z|^2)\left|\frac{h'(z)}{h(z)}\right|\left(1+\left|\mu(z)\right|\right).
\end{split}
\end{equation*}
We refer it as the \emph {log-harmonic Bloch norm} and the elements of the log-harmonic Bloch space are called log-harmonic Bloch functions.

Now we will show that $\mathcal{B}_{Lh}$ has the affine and linear invariance. To do this, we let
$$\mathcal{\phi}_{\alpha}(z)=\frac{z+\alpha}{1+\overline{\alpha}z},\quad z\in\ID,
$$
where $|\alpha|<1$.

\begin{proposition}
If $f(z)=h(z)\overline{g(z)}\in\mathcal{B}_{Lh}$, then
\begin{enumerate}
\item[{\rm(i)}] $f^{a}\overline{f}^{b}\in \mathcal{B}_{Lh}$ for any $a,b\in\mathbb{C}$\, {\rm(affine invariance)}
\item[{\rm(ii)}] $f\circ \mathcal{\phi}_{\alpha}\in \mathcal{B}_{Lh}$ for any $\alpha\in \ID$\, {\rm(linear invariance).}
\end{enumerate}
\end{proposition}
\begin{proof}
For the proof of (i), we let $f=h\overline{g}$, and consider
$$F=f^{a}\overline{f}^{b}=h^{a}g^{b}\overline{h^{\overline{b}}g^{\overline{a}}}.
$$
Elementary computations  give
\begin{equation*}
\begin{split}
\beta(F)&=\sup_{z\in\ID}(1-|z|^2)\left(\left|\frac{\left(h^{a}(z)g^{b}(z)\right)'}{h^{a}(z)g^{b}(z)}\right|
+\left|\frac{\big(h^{\overline{b}}(z)g^{\overline{a}}(z)\big)'}{h^{\overline{b}}(z)g^{\overline{a}}(z)}\right|\right)\\
&=\sup_{z\in\ID}(1-|z|^2)\left(\left|a\frac{h'(z)}{h(z)}+b\frac{g'(z)}{g(z)}\right|
+\left|\overline{b}\frac{h'(z)}{h(z)}+\overline{a}\frac{g'(z)}{g(z)}\right|\right)\\
&\leq
\left(|a|+|b|\right)\beta(f)<+\infty.
\end{split}
\end{equation*}
By Definition \ref{dfLHBF}, the desired assertion follows.

For the proof of (ii), we write $F=f\circ \mathcal{\phi}_{\alpha}=H\overline{G}$ so that
\begin{equation*}
\frac{H'(z)}{H(z)}=\frac{h'\left(\mathcal{\phi}_{\alpha}(z)\right)}{h\left(\mathcal{\phi}_{\alpha}(z)\right)}
\cdot\frac{1-|\alpha|^2}{(1+\overline{\alpha}z)^2}\quad{\rm and}\quad
\frac{G'(z)}{G(z)}=\frac{g'\left(\mathcal{\phi}_{\alpha}(z)\right)}{g\left(\mathcal{\phi}_{\alpha}(z)\right)}
\cdot\frac{1-|\alpha|^2}{(1+\overline{\alpha}z)^2}.
\end{equation*}
Consequently,
\begin{equation*}
\begin{split}
\beta(F)
&=\sup_{z\in\ID}\frac{(1-|z|^2)(1-|\alpha|^2)}{|1+\overline{\alpha}z|^2}
\left(\left|\frac{h'\left(\mathcal{\phi}_{\alpha}(z)\right)}{h\left(\mathcal{\phi}_{\alpha}(z)\right)}\right|+
\left|\frac{g'\left(\mathcal{\phi}_{\alpha}(z)\right)}{g\left(\mathcal{\phi}_{\alpha}(z)\right)}\right|\right)\\
&=\sup_{z\in\ID}(1-|\mathcal{\phi}_{\alpha}(z)|^2)
\left(\left|\frac{h'\left(\mathcal{\phi}_{\alpha}(z)\right)}{h\left(\mathcal{\phi}_{\alpha}(z)\right)}\right|+
\left|\frac{g'\left(\mathcal{\phi}_{\alpha}(z)\right)}{g\left(\mathcal{\phi}_{\alpha}(z)\right)}\right|\right),
\end{split}
\end{equation*}
which gives that $\beta(F)=\beta(f)$. The proof is complete.
\end{proof}

 \vskip.20in

\begin{center}{\sc Acknowledgements}
\end{center}

\vskip.05in
The work of the second author is supported  in part by Mathematical Research Impact Centric Support
(MATRICS) grant, File No.: MTR/2017/000367, by the Science and Engineering Research Board (SERB),
Department of Science and Technology (DST), Government of India. The work of Zhihong Liu
was completed during his visit to Indian Statistical Institute, Chennai Centre.
The authors would like to thank the referees for their valuable suggestions and comments which essentially improved the quality of this paper.




\end{document}